\documentclass[12pt]{article}

\usepackage{amssymb,amsmath}
\usepackage{amsbsy}
\usepackage{amsthm}
\usepackage{mathrsfs}
\usepackage{verbatim}
\usepackage[colorlinks]{hyperref}
\usepackage{graphicx,subfigure}
\usepackage[english]{babel}
\usepackage{dsfont}
\usepackage[utf8]{inputenc}
\usepackage{xifthen}

\usepackage{fullpage}

\usepackage{tikz}
\usetikzlibrary{backgrounds,fit, matrix}
\usetikzlibrary{positioning}
\usetikzlibrary{calc,through,chains}
\usetikzlibrary{arrows,shapes,snakes,automata, petri}

\usepackage{authblk}

\newtheorem{theorem}{Theorem}[section]

\newtheorem{Theorem}[theorem]{Theorem}
\newtheorem{Corollary}[theorem]{Corollary}
\newtheorem{Lemma}[theorem]{Lemma}
\newtheorem{Proposition}[theorem]{Proposition}

\theoremstyle{definition}
\newtheorem{Definition}[theorem]{Definition}
\newtheorem{Example}[theorem]{Example}

\theoremstyle{remark}

\newtheorem{Question}[theorem]{Question}

\newtheorem{Remark}[theorem]{Remark}

\numberwithin{figure}{section}

\newcommand{\PP}{{\mathbb P}}

\newcommand{\Z}{{\mathbb Z}}

\newcommand{\R}{{\mathbb R}}

\newcommand{\mc}[1]{\mathcal{#1}}

\newcommand{\mt}[1]{\text{#1}}

\newcommand{\ndec}[1][]{\ifthenelse{\isempty{#1}}{nondecomposable}{#1-nondecomposable}}
\newcommand{\bun}[4][]{#2 \to #3 \xrightarrow{#1} #4}

\begin{document}

\title{Monodromy of torus fiber bundles and decomposability problem}

\author[1]{Mahir Bilen Can}
\author[2]{Mustafa Topkara\footnote{This author is partially supported by Mimar Sinan Fine Arts University Scientific Research Projects Unit.}}

\affil[1]{{\small Tulane University, New Orleans; mcan@tulane.edu}}
\affil[2]{{\small Mimar Sinan Fine Arts University; m.e.topkara@gmail.com}}
\normalsize

\date{\today}
\maketitle

\begin{abstract}

The notion of a (stably) decomposable fiber bundle is introduced. In low dimensions, for torus fiber 
bundles over a circle the notion translates into a property of elements of the special linear 
group of integral matrices. We give a complete characterization of the stably decomposable 
torus fiber bundle of fiber-dimension less than 4 over the circle. 
\vspace{.5cm}

\end{abstract}

\section{Introduction}\label{S:Introduction}

A closed oriented manifold $M$ is said to be decomposable if it is homeomorphic to $M_1 \times M_2$, 
where $M_1$ and $M_2$ are manifolds with $0< \dim M_i < \dim M$ for $i=1,2$. 
The question of which manifolds are decomposable carries an obvious importance for mathematics. 
In dimension 1, up to homeomorphism, the only closed and oriented manifold is the unit circle $S^1$, 
so, we do not have much ado. The 2-torus $T^2=S^1\times S^1$, is decomposable, however, the 
2-sphere $S^2$ is not homeomorphic to a product of two manifolds of dimension 1, hence, 
it is not decomposable. We call such manifolds indecomposable.

It is easily seen by using elementary properties of Euler characteristic that any closed and oriented surface of 
genus $g\neq 1$ is indecomposable. However, already in dimension 3 detecting indecomposable manifolds 
requires some deeper information about the fundamental group of the underlying manifold, and in dimension 
4 the problem becomes significantly harder. 
Moreover, in the astounding world of topological manifolds it may happen that when an indecomposable 
manifold $M$ is crossed (cartesian product) with another (indecomposable) manifold $N$ with 
$1\leq \dim N \leq \dim M$ the resulting manifold decomposes into lower dimensional pieces:
\begin{align*}
M\times N \simeq M_1\times \cdots \times M_r,
\end{align*}
where $M_1,\dots, M_r$ are manifolds with $0< \dim M_i < \dim M$ for $i=1,\dots,r$.
In this case, $N$ is said to stably decompose $M$, and $M$ is called stably decomposable.
If there is no such $N$, then $M$ is called stably indecomposable.

In their recent preprint~\cite{KwasikSchultz}, Kwasik and Schultz show that a 3-manifold $M$ is 
decomposable if and only if it is stably decomposable. 
Moreover, they show that there exist infinitely many non-homeomorphic indecomposable 4-manifolds 
$M$ such that $M\times S^k$ is homeomorphic to $S^1 \times S^k \times \R\PP^3$, where $k=2,3,4$.

\vspace{.5cm}

The notion of indecomposability applies to different categories. 
One could consider the decomposability with respect to diffeomorphism,
furthermore, one could confine oneself to homogeneous manifolds only.
Here, a homogenous manifold means one that admits a transitive action of a Lie group, hence
is isomorphic to a coset space. A homogeneous manifold $M$ is called homogeneously decomposable 
if $M$ is diffeomorphic to a direct product $M_1\times M_2$ of homogeneous manifolds 
$M_1$ and $M_2$ with $\dim M_i \geq 1$.

Recall that a manifold is called aspherical if its homotopy groups $\pi_k(M)$ vanish for $k\geq 2$. 
An important class of aspherical manifolds is the class of solvmanifolds. 
By definition, a {\it solvmanifold} is a homogeneous space of a connected solvable Lie group. 
It turns out that two compact solvmanifolds are isomorphic if and only if their fundamental groups are isomorphic, \cite{Mostow}.
If $M$ is an aspherical compact manifold with a solvable fundamental group, 
then it is homeomorphic to a solvmanifold. This follows from the work of Farrell and Jones 
[Corollary B,~\cite{FarrellJones}] for $\dim \neq 3$, and from Thurston's Geometrization Conjecture
for $\dim = 3$. 
It follows from these results that if a compact manifold admits a torus fibration over a torus, then it is homeomorphic to a solvmanifold. 

In general, the notion of homogeneous decomposability is different from decomposibility in differential category. 
However, it is shown by Gorbatsevich (Proposition 1,~\cite{Gorbatsevich7}) that homogenous decomposability is 
equivalent to (topological) decomposability as defined in the first paragraph.

\vspace{1cm}

The goal of our paper is to define decomposability and carry out a similar analysis 
in the category of fiber bundles, where fiber product plays the role of the cartesian 
product. Roughly speaking, we call a bundle over a space $B$ as $B$-indecomposable 
if it is indecomposable in the category of fiber bundles over $B$. This generalizes the previous 
concepts since the (stable) decomposability of a manifold $M$ is equivalent to (stable) 
$*$-decomposability of the trivial bundle $\bun{M}{M}{*}$ over the one point space $*$.

In Section~\ref{S:decomposability}, we observe that (stable) indecomposability of the fiber 
implies (stable) $B$-indecomposability, hence by the result of Kwasik and Schultz, 
there exist 4-fiber-dimensional fiber bundles which are $B$-indecomposable but stably 
$B$-decomposable. 
In Section~\ref{S:prep}, for low dimensions we interpret the stable decomposability of a torus fiber bundle over 
a circle in terms of a matrix problem in $\mt{SL}(n,\Z)$. 
In Section~\ref{S:3x3}, we show that, for a 3-torus bundle over $S^1$, the $S^1$-indecomposability 
(Theorem~\ref{T:our remark}) and the stably $S^1$-indecomposability (Theorem~\ref{T:stably T3}) 
are both determined by the characteristic polynomial of their monodromy matrices.  
We end our paper by indicating how our results give information about the same problem
for torus fibrations over an arbitrary torus as the base space (Remark~\ref{R:bitirici}).

\section{Preliminaries}

Standard references for theory of fiber bundles are the textbooks~\cite{Husemoller} and \cite{Steenrod}.

\vspace{.5cm}

For us, a {\it fiber bundle} is a triplet $(F,E,B)$ of topological spaces along with a surjective map 
$p: E\rightarrow B$ satisfying the following properties: 
\begin{itemize}
\item $F$ and $E$ are closed, oriented manifolds;
\item $B$ is a connected CW-complex;
\item for each point $x\in E$ there exists a 
neighborhood $U_x\subseteq B$ of $p(x)$ and a homeomorphism $\phi_x: p^{-1}(U_x)\rightarrow U_x\times F$ 
such that the following diagram commutes
\begin{figure}[htp]
\begin{center}
\begin{tikzpicture}
\node (a) at (-2,2) {$p^{-1}(U_x)$};
\node (b) at (2,2) {$U_x\times F$};
\node (c) at (2,0) {$U_x$};
\draw[->] (a) -- (b);
\draw[->] (b) -- (c);
\draw[->] (a) -- (c);
\node at (2.45,1) {$pr_1$};
\node at (0,.65) {$p$};
\node at (0.25,2.35) {$\phi_x$};
\end{tikzpicture}
\end{center}
\end{figure}

Here, $pr_1$ denotes the first projection.
\end{itemize}

When we have a fiber bundle, we call $E$ the total space, $B$ the base, and $F$ the fiber. 
Note that for any $b\in B$, the preimage $p^{-1}(b)$ is homeomorphic to $F$. 
We call $p$ the projection map of the bundle. 

If $F,E,B$ and $p$ are given, then we denote the corresponding fiber bundle by $\bun[p]{F}{E}{B}$ 
and occasionally skip $p$ from our notation. If there is no risk of confusion we denote $\bun[p]{F}{E}{B}$ by $E$ only.

If $\bun[p]{F}{E}{B}$ and $\bun[p']{F'}{E'}{B}$ are two fiber bundles over the same base, 
then a {\em fiber bundle map} from $E$ to $E'$ is a continuous map $\psi: E\to E'$ such that 
$p=\psi \circ p'$. If furthermore $\psi$ is a homeomorphism, then the fiber bundles $E$ and $E'$
are called isomorphic fiber bundles.

By the {\it fiber dimension} of a fiber bundle $E$, we mean the dimension of the fiber $F$ and denote it by 
$$
\dim_B E := \dim F.
$$ 
By the {\em fiber product} of two fiber bundles $\bun[p]{F}{E}{B}$ and $\bun[p']{F'}{E'}{B}$ over the same base $B$,
we mean the fiber bundle $\bun[q]{F\times F'}{E\times_B E'}{B}$, where
$$
E\times_B E' := \{ (x,y)\in E\times E':\ p(x)=p'(y) \}, 
$$
and the projection map defined by $q(x,y) := p(x)=p'(y)$ for $(x,y)\in E\times_B E'$.

Fiber product is an associative operation on fiber bundles, so we do not use parentheses in our notation.

\section{Decomposability of fiber bundles}\label{S:decomposability}

\begin{Definition}
A fiber bundle $\bun F E B$ is called $B$-decomposable 
if there exist fiber bundles $\bun{F_1}{E_1}{B}$ and $\bun{F_2}{E_2}{B}$ such that 
\begin{enumerate}
\item $\dim_B E_1,\dim_B E_2 < \dim_B E$; 
\item $E\cong E_1 \times_B E_2$ (fiber bundle isomorphism).
\end{enumerate}
If $E$ is not a $B$-decomposable fiber bundle, then we call it $B$-indecomposable. 
\end{Definition}
\begin{Definition}
Let $\bun{F}{E}{B}$ be a bundle of fiber-dimension $n$. 
$E$ is called stably $B$-decomposable if there exists a bundle $\bun{F'}{E'}{B}$ of fiber-dimension $0<n'\leq n$ such that 
$$
E \times_B E' \cong E_1\times_B \cdots \times_B E_k
$$ 
for some bundles $E_1,\dots,E_k$ with base $B$ and $0<\dim_B E_i < n$
for $i=1,\ldots k$. In this case, $E'$ is said to stably $B$-decompose $E$, 
and $E$ is called stably $B$-decomposable. 
If $E$ is not stably $B$-decomposable by any fiber bundle over the base $B$, 
then $E$ is called stably $B$-indecomposable. 
\end{Definition}

Since a decomposition of a fiber bundle results in a decomposition of its fiber, if 
$F$ is an indecomposable manifold, then there is no $B$-decomposable fiber bundle 
of the form $\bun{F}{E}{B}$.
Similarly, the corresponding statement holds true if $F$ is assumed to be stably indecomposable. 
We record this observation in the next proposition without a proof.

\begin{Proposition} \label{p:ndec}
The fiber bundle $\bun[p]{F}{E}{B}$ is $B$-indecomposable (resp. stably $B$-indecomposable) 
if $F$ is indecomposable (resp. stably indecomposable). 
\end{Proposition}

By combining the previous proposition with the results of Kwasik and Schultz~\cite{KwasikSchultz}
we obtain the following corollary.

\begin{Corollary}\label{C:converse is wrong}
Let $\bun{F}{E}{B}$ be a fiber bundle of fiber-dimension $n$. Then
\begin{enumerate}
\item for $n \leq 3$, if $F$ is indecomposable, then $E$ is (stably) $B$-indecomposable. 
\item There exists a fiber bundle of fiber-dimension $4$ which is $B$-indecomposable but stably $B$-decomposable.
\end{enumerate}
\end{Corollary}

In the light of Corollary~\ref{C:converse is wrong} the following question deserves an answer. 
\begin{Question}
Is it possible to prove a result that is similar to the second part of Corollary~\ref{C:converse is wrong} if we insist on smaller fiber dimensions than 4? 
In other words, is there a fiber bundle $\bun{F}{E}{B}$ with $\dim_B E \leq 3$ which is $B$-indecomposable but stably $B$-decomposable?
\end{Question}

We resolve this problem in Section~\ref{S:3x3}.

\section{Torus fibrations over $S^1$}\label{S:prep}

We are going to show in Example~\ref{E:smallest counter example} that the converse 
of Proposition~\ref{p:ndec} is not true.
This is one of the reasons why this project is worthwhile to consider. To better explain, 
we are going to have a preliminary (standard) discussion about the classification of fiber bundles 
on relatively simple spaces. For this purpose, we identify the circle $S^1$ as a quotient of the closed 
interval $[0,1]$ via a map $q:[0,1]\to S^1$ such that $q(0)=q(1)$ and which is one-to-one on the open 
interval $[0,1]-\{0,1\}$.

Let $F$ be a space and $\tau : F\to F$ be an orientation preserving homeomorphism. We consider the trivial $F$-bundle structure
on the cartesian product 
$$
\bun[pr_1]{F}{[0,1]\times F}{[0,1]}.
$$ 
By gluing the two ends of this bundle by $\tau$, let us define the new bundle $\bun[p]{F}{E_\tau}{S^1}$. 
Thus $E_\tau$ is the quotient space $([0,1]\times F)/\sim$, where $\sim$ is the equivalence relation
$$
(0,a)\sim(1,\tau(a)) \ \text{ for } a \in F,
$$
and $p$ is the projection that maps an equivalence class $\overline{(t,a)}\in E_\tau$ to the point $q(t)$ in $S^1$. 
The total space $E_\tau$ is called the mapping torus of $\tau$, and $\tau$ is called the monodromy map of the bundle $E_\tau$.
This construction yields isomorphic bundles if two homeomorphisms are isotopic to each other.
Indeed, if $G:[0,1]\times F \to F$ is an isotopy from $\tau$ to $\rho$, then the isomorphism between $E_\tau$ and $E_\rho$ is given by 
$$
\overline{(t,a)}\mapsto \overline{(t,G(t,\tau^{-1}(a))}.
$$
Similarly, if two homeomorphisms $\tau$ and $\rho$ are conjugate, then they define isomorphic bundles.
(The reason is that both of the homeomorphisms $\tau \circ \rho$ and $\rho \circ \tau$ result in mapping tori
which are isomorphic as fiber bundles over $S^1$.)

Conversely, any fiber bundle $\bun[p]{F}{E}{S^1}$ on the circle can be pulled back onto $[0,1]$ 
via the map $q:[0,1]\to S^1$. Since $[0,1]$ is contractible, the pullback bundle $E'=q^*E$ is 
isomorphic to the product bundle via some map $\Phi:[0,1]\times F \to E'$. 
Now, the fiber bundle $E$ can be recovered from $E'$ by the procedure described above 
using the monodromy map $\tau:F \to F$ defined by the relation $\Phi(0,a)=\Phi(1,\tau(a))$.

As a consequence of the above construction, we see that there is a correspondence between 
the conjugacy classes of the mapping class group of $F$ and isomorphism classes of fiber 
bundles on $S^1$ with fiber $F$.

We apply these considerations to torus bundles. Let us fix an ordered basis $\mathcal{B}$ for $H_1(T^n)$. 
For a homeomorphism $\tau: T^n \to T^n$, 
the \emph{monodromy matrix} $\mathcal{M}_\tau \in \mt{SL}(n,\Z)$ of the bundle $\bun[]{T^n}{E_\tau}{S^1}$ is the matrix representing the induced map 
$\tau_*: H_1(T^n) \to H_1(T^n)$ in the basis $\mathcal{B}$, i.e.
\[
\mathcal{M}_\tau=[\tau_*]_\mathcal{B}.
\]
If the fiber bundles $E_\tau$ and $E_\rho$ are isomorphic, then $\tau$ and $\rho$ are conjugate hence their monodromy matrices $\mathcal{M_\tau}$ and $\mathcal{M_\rho}$ are conjugate too. The converse also holds for low dimensions:

\begin{Lemma}\label{L:torus}
	Let $T^n$ denote the $n$-torus $\times_{i=1}^n S^1$ for $n\leq 3$. Then the fiber bundles over $S^1$ 
	with fiber $T^n$ are parametrized by the conjugacy classes in $\mt{SL}(n,\Z)$.
\end{Lemma}

\begin{proof}
	The mapping class group of the 2-torus is well known to be equal to $\mt{SL}(2,\Z)$, see~\cite{BensonFarb}. 
	$n=3$ case follows from~\cite[Corollary 7.5]{W} and the fact that $\mt{Out}(\pi_1(T^3))$ is isomorphic to $\mt{GL}(3,\Z)$. 
	(The subgroup $\mt{SL}(3,\Z)$ corresponds to orientation preserving self-homeomorphisms up to isotopy.) 
\end{proof}

\begin{Proposition}\label{P:first observation}
	Let $E$ be an $n$-torus bundle on $S^1$ for $n\leq 3$. $E$ is $S^1$-decomposable if and only if the monodromy of $E$ is given by 
	a matrix $M\in \mt{SL}(n,\Z)$ that is similar to a block diagonal matrix 
	$A\oplus B\in \mt{SL}(n,\Z)$, where $(A,B)\in \mt{SL}(k,\Z)\times \mt{SL}(n-k,\Z)$ for some $k\in \{1,\dots, n-1\}$. 
	In this case, if $E'$ and $E''$ are the $k$- and $(n-k)$- torus bundles such that $E\cong E'\times_{S^1} E''$, then 
	the monodromy of $E'$ and $E''$ are given by the matrices $A\in \mt{SL}(k,\Z)$ and $B\in \mt{SL}(n-k,\Z)$, respectively. 
\end{Proposition}
\begin{proof}
	
	The second claim is immediate from the first one. 
	
	An $n$-torus bundle $E$ is $S^1$-decomposable if and only if 
	there are fiber bundles $\bun{F_2}{E_1}{S^1}$ and $\bun{F_2}{E_2}{S^1}$ such that 
	$E \cong E_1\times_{S^1} E_2$ and $T^n \cong F_1\times F_2$. 
	In this case, since $\pi_m ( F_1\times F_2 ) = \pi_m (F_1) \times \pi_m(F_2)$ for all $m\geq 0$, the 
	the fundamental group of $F_1$ is isomorphic to $\Z^k$ for some $k\in \{1,\dots, n-1\}$, hence $\pi_1(F_2) \cong \Z^{n-k}$,
	and $\pi_m (F_1) = \pi_m (F_2) = 0$ for all $m\geq 2$.

	By the homotopy classification of aspherical spaces, we know that if all higher homotopy groups $\pi_m$ ($m\geq 2$) of
	two manifolds $M_1$ and $M_2$ vanish, then $M_1$ is homotopy equivalent to $M_2$ if and only if $\pi_1(M_1)= \pi_1(M_2)$. 
	(See [Theorem 2.1~\cite{Luck}]).
	It follows that $F_1$ is homotopy equivalent to the $k$-torus and $F_2$ is homotopy equivalent to $n-k$-torus. 
	It is shown by Hsiang and Wall in~\cite{HsiangWall} that tori are topologically rigid. In other words, 
	if $N$ is a manifold which is homotopy equivalent to $T^m$, then $N$ is actually homeomorphic to $T^m$. 
	Thus, $F_1 \cong T^k$ and $F_2 \cong T^{n-k}$. 
	Summarizing,  $E$ is $S^1$-decomposable if and only if $E$ is a fiber product of two torus bundles.
	The result now is a consequence of Lemma~\ref{L:torus}.
	
\end{proof}

\begin{Example}\label{E:smallest counter example}
	The matrix
	$$
	M=\begin{pmatrix}
	1& 1 \\
	0 & 1
	\end{pmatrix} \in \mt{SL}(2,\Z)
	$$
	is not diagonalizable. Therefore, the nontrivial 2-torus bundle over $S^1$ that is determined by $M$ is $S^1$-indecomposable.
	Indeed, a 2-torus bundle over $S^1$ is $S^1$-decomposable if and only if its monodromy matrix is the identity matrix in 
	$\mt{SL}(2,\Z)$. 
	In contrast with Corollary~\ref{C:converse is wrong}, $F=T^2$ is a decomposable manifold.
\end{Example}

\vspace{.5cm}

Let us fix a matrix $M\in \mt{SL}(n,\Z)$ and identify the $n$-dimensional torus $T^n$ with $\R^n / \Z^n$. 
We consider a homeomorphism $\phi_M:T^n \to T^n$ such that for any
$x=\begin{pmatrix}
x_1 \\
\vdots \\
x_n
\end{pmatrix} \in \R^n$, the value of $\phi_M$ at $\bar{x}\in T^n$ is
$\phi_M(\bar{x}):=\overline{Mx}$.
In the basis of $H_1(T^n)$ that is formed by the classes of maps $\gamma_i:[0,1] \to T^n$, $i=1,\dots,n$ 
defined by
$$
\gamma_i(t)=
\begin{pmatrix}
0 \\
\vdots \\
t \\
\vdots \\
0
\end{pmatrix}\qquad (\text{$t$ is in the $i$th row}), 
$$
the monodromy matrix $\mathcal{M}_{\phi_M}$ of $\phi_M$ is represented by $M$. 
In the following discussion, a fiber bundle with the monodromy map $\phi_M$ is denoted by $E_M$.
\begin{Lemma} \label{L:SumOfLow}
	For two matrices $M,M' \in \mt{SL}(n,\Z)$, the corresponding fiber bundles $E_M$ and $E_{M'}$ are isomorphic if and only if $M$ and $M'$ are conjugate in $\mt{SL}(n,\Z)$.
\end{Lemma}

\begin{proof}
	Let $E_M$ and $E_M'$ be isomorphic. By the discussion preceding Lemma~\ref{L:torus}, $M$ is conjugate to $M'$ in $\mt{SL}(n,\Z)$.
	
	Now, let us assume that $M$ is conjugate to $M'$ and hence there exists $P \in \mt{SL}(n,\Z)$ such that $M'=P^{-1}MP$. Then the monodromy map of $E_{M'}$ is $\phi_{P^{-1}MP}$ which can be seen to be equal to $\phi_P^{-1}\phi_M\phi_P$ from the construction of the maps $\phi_M$, and hence is a conjugate of the monodromy map of $E_M$. Therefore, $E_M$ and $E_{M'}$ are isomorphic.
\end{proof}

Observe that for $M \in \mt{SL}(n_1,\Z)$ and $N \in \mt{SL}(n_2,\Z)$, we have $\phi_{M\oplus N}=\phi_M \times \phi_N : T^{n_1}\times T^{n_2} \to T^{n_1}\times T^{n_2}$ and hence $E_{M\oplus N}=E_M\times_{S^1}E_N$. Using this observation, we obtain the following proposition which tells us that we can extend Lemma~\ref{L:torus} to fiber sums of low dimensional torus bundles on the circle and identify them by their monodromy matrices.

\begin{Proposition}\label{P:decomposable iff}
	Let 
	$E=E_1\times_{S^1}\cdots\times_{S^1}E_k$ and
	$E'=E'_1\times_{S^1}\cdots\times_{S^1}E'_l$, where each $E_i$ (resp. $E'_i$) is a $T^{n_i}$ (resp. $T^{n'_i}$) bundle on $S^1$. 
	Further assume that $n_i\leq 3$ for $i \in 1,\ldots, k$, $n'_i\leq 3$ for $i \in 1,\ldots, l$ and $\sum_{i=1}^k n_i=\sum_{i=1}^l n'_i:=n$. 
	Then the fiber bundles of $E$ and $E'$ are isomorphic if and only if their monodromy matrices are conjugate in $\mt{SL}(n,\Z)$.
\end{Proposition}

\begin{proof}
	If the bundles $E$ and $E'$ are isomorphic then their monodromy matrices are isomorphic by the discussion preceeding Lemma~\ref{L:torus}.
	
	Conversely, let us assume that the monodromy matrices of $E$ and $E'$ are conjugate in $\mt{SL}(n,\Z)$. By Lemma~\ref{L:torus}, for each $i=1,\ldots,k$ there exists a matrix $M_i \in \mt{SL}(n_i,\Z)$ such that $E_i$ is isomorphic to $E_{M_i}$. Therefore, by the observation preceeding this proposition, we have that $E=E_1\times_{S^1}\cdots\times_{S^1}E_k$ is isomorphic to
	$E_{M_1}\times_{S^1}\cdots\times_{S^1}E_{M_k}=E_{M_1\oplus\cdots\oplus M_k}$. Similarly, for each $i=1,\ldots,l$ there exists matrices $M'_i \in \mt{SL}(n'_i,\Z)$ such that $E'$ is isomorphic to $E_{M'_1\oplus\cdots\oplus M'_l}$. Thus, as a result of Lemma~\ref{L:SumOfLow}, the fiber bundles are isomorphic if and only if their monodromy matrices are isomorphic.
\end{proof}

\subsection{Stable decomposability when $\dim F =2$}

\begin{Lemma}\label{L:2-torus case}
If $E$ is an $S^1$-indecomposable $T^2$-bundle, then $E$ is stably $S^1$-indecomposable. 
\end{Lemma}

\begin{proof}
Towards a contradiction assume that $E$ is a stably $S^1$-decomposable $T^2$-bundle. 
Then there exists a stably $S^1$-indecomposable $\bun{F'}{E'}{B}$ such that $\dim_{B}E' \leq 2$
and $E$ is stably decomposed by $E'$.

We proceed with the case $\dim_B E'=2$. 
There exists fiber bundles $\bun{F_i}{E_i}{B}$ for $i=1,\dots, 4$ such that 
$$
E\times_B E' \cong E_1\times_B \cdots \times_B E_4\  \text{  and $\dim_B E_i = 1$ for $i=1,\dots, 4$}. 
$$
In particular, $E_i$'s are $S^1$ bundles, hence, $F\times F' = T^4$. 
A special case of the argument that is made in the proof of Proposition~\ref{P:first observation} 
shows that $F'$ is homeomorphic to $T^2$. 
Since the monodromy matrix of $E\times_B E'$ is the $4\times 4$ diagonal matrix and since the block diagonal 
matrices $A_1\oplus \cdots \oplus A_k$, $B_1\oplus \cdots \oplus B_k$ in $\mt{SL}(n,\Z)$ are similar if and only if 
$A_i \sim B_i$ for each $i=1,\dots, k$, we see that the monodromy matrix of $E$ is similar to a diagonal matrix. 
This contradicts with indecomposability of $E$.

The proof of the case $\dim_B E'=1$ is similar.  
\end{proof}

\begin{Proposition}
If $\bun{F}{E}{S^1}$ is $S^1$-indecomposable and $\dim_{S^1}E=2$, then it is stably $S^1$-indecomposable. 
\end{Proposition}
\begin{proof}
Assume towards a contradiction that $\bun{F}{E}{B}$ is stably $S^1$-decomposable.
Then there is a stably $S^1$-indecomposable fiber bundle $\bun{F'}{E'}{B}$ of fiber-dimension $\leq 2$ stably decomposing $E$. 
We prove our claim for $\dim F'=2$ only, since the proof of the case $\dim F'=1$ is similar. 
Then there exists fiber bundles $\bun{F_i}{E_i}{B}$ for $i=1,\dots, 4$ such that 
$$
E\times_B E' \cong E_1\times_B \cdots \times_B E_4\  \text{  and $\dim_B E_i = 1$ for $i=1,\dots, 4$}. 
$$
In particular, $E_i$'s are $S^1$ bundles, hence, $F\times F' = T^4$. 
A special case of the argument that is made in the proof of Proposition~\ref{P:first observation} shows that 
both of the fibers $F$ and $F'$ are homeomorphic to $T^2$. 
The rest of the proof of our claim now follows from Lemma~\ref{L:2-torus case}.

\end{proof}

\section{Indecomposable torus bundles over $S^1$}\label{S:3x3}

{\it Notation:} If there is no danger of confusion, we denote a matrix whose entries are all 0's by 
the bold-face $\mathbf{0}$ and we avoid mentioning its order.

Although the classification of conjugacy classes in $\mt{SL}(n,\Z)$ appears to be an open problem for general $n$, 
the similarity question for $3\times 3$ matrices is completely answered by the paper 
\cite{AppelgateOnishi} when the characteristic polynomial is reducible, and by \cite{Grunewald} and \cite{GrunewaldSegal}
when the characteristic polynomial is irreducible.

Let $A$ be a $3\times 3$ unimodular matrix with characteristic polynomial $f(t)\in \Z[t]$. 
We assume that $f(1)=0$, hence 1 is an eigenvalue of $A$.
By Theorem III.12~\cite{Newman}, we know that there exists $R\in \mt{SL}(3,\Z)$ such that 
$$
RAR^{-1} 
=
\begin{pmatrix}
1 & \mathbf{a} \\
\mathbf{0} & A_2
\end{pmatrix},
$$
where $A_2$ is a $2\times 2$ matrix from $\mt{SL}(2,\Z)$ and $\mathbf{a}$ is a row vector $a= ( a_1,a_2)$. 
Let us assume that the characteristic polynomial $g(t)= t^2 + \tau t + \delta \in \Z[t]$ of $A_2$ does not have any integer roots. 
Since the product of all eigenvalues of $A$ is $\det A=1$, we have $\delta = 1$. 
Now, since $g(t)$ has no integer roots, $\tau \neq 2$. Otherwise, $g(t) = (t+1)^2$. 
Let $A_0$ denote the matrix
$$
A_0=A_2 - (\tau -1) \mathbf{I}_2.
$$
Here, $\mathbf{I}_2$ is the $2\times 2$ identity matrix. 
Clearly, $A_0$ is invertible. ($\tau -1$ is not an eigenvalue for $A_2$.)

Now, let $B\in \mt{SL}(3,\Z)$ denote the matrix 
$$
B
=
\begin{pmatrix}
1 & \mathbf{b} \\
\mathbf{0} & B_2
\end{pmatrix}.
$$
It is easy to check that if $A_2$ is not similar to $B_2$ in $\mt{SL}(2,\Z)$, then $A$ and $B$ cannot be similar. 
On the other hand, if $R_2 A_2 R_2^{-1} = B_2$ for some $R_2\in \mt{GL}(2,\Z)$, then 
$$
\begin{pmatrix}
1 & \mathbf{0} \\
\mathbf{0} & R_2
\end{pmatrix}^{-1} 
\begin{pmatrix}
1 & \mathbf{b} \\
\mathbf{0} & B_2
\end{pmatrix}
\begin{pmatrix}
1 & \mathbf{0} \\
\mathbf{0} & R_2
\end{pmatrix}
=
\begin{pmatrix}
1 & \mathbf{b} R_2 \\
\mathbf{0} & A_2
\end{pmatrix}.
$$
Therefore, as far as the similarity of $A$ and $B$ is concerned, there is no harm in assuming that $A_2=B_2$.

\begin{Lemma}[Appelgate-Onishi, Lemma 22,\cite{AppelgateOnishi}]\label{L:AO}
Let 
$$
A
=
\begin{pmatrix}
1 & \mathbf{a} \\
\mathbf{0} & A_2
\end{pmatrix} \
\text{ and } \ 
B
=
\begin{pmatrix}
1 & \mathbf{b} \\
\mathbf{0} & A_2
\end{pmatrix}
$$
be two elements from $\mt{SL}(3,\Z)$ and let $\tau$ denote the coefficient of $t$ in the 
characteristic polynomial $g(t)\in \Z[t]$ of $A_2$. Let $m=\tau-2$ and set
$A_0=A_2 - (\tau -1) \mathbf{I}_2$.
Then 
$A$ is similar to $B$ if and only if there exist $R_2\in \mt{SL}(2,\Z)$ and $u = \pm 1$ such that 
\begin{enumerate}
\item $R_2$ commutes with $A_2$;
\item $\mathbf{b} A_0 R_2 \equiv u \mathbf{a} A_0 \mod m$.
\end{enumerate}
\end{Lemma}

\vspace{.5cm}

\begin{Theorem}\label{T:our remark}
Let $\bun{T^3}{E}{S^1}$ be a $T^3$-bundle over $S^1$, $A\in \mt{SL}(3,\Z)$ denote its monodromy matrix,
and let $f(t)\in \Z[t]$ denote the characteristic polynomial of $A$. Then 
$E$ is $S^1$-indecomposable if and only if one of the following three distinct cases occurs:
\begin{enumerate}
\item $f(t)$ is irreducible over $\Z$.
\item $f(-1)=0$ but $f(1)\neq 0$. 
\item $f(1)=0$ and the monodromy is of the form 
$
A = 
\begin{pmatrix}
1 & \mathbf{a} \\
\mathbf{0} & A_2
\end{pmatrix}
$
and it satisfies 
\begin{align*}
\mathbf{0} \not\equiv \mathbf{a} (A_2 + \mathbf{I}_2) \mod (\mt{Tr}(A_2)+2).
\end{align*} 
\end{enumerate}
\end{Theorem}

\begin{proof}

It is clear that all of the listed cases are distinct and for any $A\in \mt{SL}(3,\Z)$ exactly one of these three cases occurs. 
For the converse, we prove that each case implies the $S^1$-indecomposability. 

First, we assume that $f(t)$ is irreducible but $E$ is $S^1$-decomposable. 
By Proposition~\ref{P:decomposable iff}, a $T^3$ fiber bundle over $S^1$ is $S^1$-decomposable if and only if its monodromy matrix is similar to
$
B=
\begin{pmatrix}
1 &  \mathbf{0} \\
 \mathbf{0} & A_2
\end{pmatrix}
$
for some $A_2 \in \mt{SL}(2,\Z)$.
Since characteristic polynomial is invariant under similarity, it follows that the characteristic polynomial of $A$ is divisible by $t-1$,
which contradicts with the irreducibility of $f(t)$. Therefore $E$ is $S^1$-indecomposable.

Next, we assume that $f(-1)=0$ and $f(1)\neq 0$. 
Once again, by Proposition~\ref{P:decomposable iff}, 1 is an eigenvalue of an $S^1$-decomposable $T^3$ bundle.
Hence, $A$ cannot be similar to the monodromy matrix of any $S^1$-decomposable $T^3$ bundle
and $E$ is $S^1$-indecomposable.

Finally, we assume that $f(1)=0$, hence the monodromy is of the form 
$
A = 
\begin{pmatrix}
1 & \mathbf{a} \\
\mathbf{0} & A_2
\end{pmatrix}
$
and it satisfies 
\begin{align*}
 \mathbf{0} \not\equiv \mathbf{a} (A_2 + \mathbf{I}_2) \mod (\mt{Tr}(A_2)+2).
\end{align*} 
(The reason that $A$ is similar to this particular form follows from Theorem III.12~\cite{Newman}.)
Now, by Lemma~\ref{L:AO} we know that 
$
A = 
\begin{pmatrix}
1 & \mathbf{a} \\
 \mathbf{0} & A_2
\end{pmatrix}
$
is similar to 
$
\begin{pmatrix}
1 &  \mathbf{0}\\
 \mathbf{0} & A_2
\end{pmatrix}
$
if and only if the congruence 
\begin{align}\label{A:a step}
 \mathbf{0} \equiv  \mathbf{a} (A_2 - (\tau -1) \mathbf{I}_2) \mod m
\end{align} 
holds. Here, $m$ is $-2$ plus the coefficient of $t$ in the characteristic polynomial $g(t)$ of $A_2$. 
We re-write this in a more suggestive form as follows: 
Since $\tau$ is $-\mt{Tr}(A_2)$, the condition (\ref{A:a step}) is equivalent to 
\begin{align*}
 \mathbf{0} \equiv \mathbf{a} A_2 + (-\mt{Tr}(A_2) -1) \mathbf{a} \mod (-\mt{Tr}(A_2)-2),
\end{align*} 
or 
\begin{align*}
 \mathbf{0} \equiv \mathbf{a} (A_2 + \mathbf{I}_2) \mod (-\mt{Tr}(A_2)-2),
\end{align*} 
which is equivalent to 
\begin{align*}
 \mathbf{0} \equiv \mathbf{a} (A_2 + \mathbf{I}_2) \mod (\mt{Tr}(A_2)+2).
\end{align*}

\end{proof}

\begin{Example}\label{R:our remark}
If $A$ denotes the matrix 
$$
A = 
\begin{pmatrix}
1 & 0 & b \\
0 & c & d \\
0 & e & f 
\end{pmatrix},
$$
then 
\begin{align*}
\mathbf{a} (A_2+\mathbf{I}_2)  &= 
(0,b) 
\begin{pmatrix}
c+1 & d \\
e & f+1
\end{pmatrix}\\
&= (be, b(f+1))
\end{align*}
and $\mt{Tr}(A_2)+2= c+f+2$.
Therefore, if $be$ and $b(f+1)$ are not divisible by $c+f+2$, then $A$ is not similar to any matrix of the form 
$$
\begin{pmatrix}
1 & \mathbf{0} \\
 \mathbf{0} & A_2
\end{pmatrix}.
$$
\end{Example}

\vspace{.5cm}

Theorem~\ref{T:our remark} is a characterization for $S^1$-indecomposability of a $T^3$-bundle.
Our next result shows that, unlike a $T^2$-bundle, for $T^3$-bundles over $S^1$, 
the $S^1$-indecomposability is not equivalent to stably $S^1$-indecomposability.

\begin{Theorem}\label{T:stably T3}
Let $E$ be a $T^3$-bundle over $S^1$, $A\in \mt{SL}(3,\Z)$ denote its monodromy matrix,
and let $f(t)\in \Z[t]$ denote the characteristic polynomial of $A$. 
Then $E$ is stably $S^1$-indecomposable if and only if one of the following two distinct cases occurs:
\begin{enumerate}
\item $f(t)$ is irreducible over $\Z$.
\item $f(-1)=0$ but $f(1)\neq 0$. 
\end{enumerate}

\end{Theorem}

\begin{proof}

Clearly, it suffices to prove our claim for $S^1$-indecomposable $T^3$-fiber bundles only. 
By Theorem~\ref{T:our remark}, there are three cases to consider. 
First we are going to show that if $A$ has the form as in 3. of the Theorem~\ref{T:our remark}, 
then $E$ is stably $S^1$-decomposable: Assume that $A\in \mt{SL}(3,\Z)$ is of the form 
$$
A = 
\begin{pmatrix}
1 & a & b \\
0 & c & d \\
0 & e & f 
\end{pmatrix}
$$
and it satisfies the congruence.
Notice 
$$
A_{2,2} := 
\begin{pmatrix}
c & d \\
e & f 
\end{pmatrix}\in \mt{SL}(2,\Z)
$$
cannot have 1 as one of its eigenvalues, so, we know that $A_{2,2}-\mathbf{I}_2$ is invertible.

Set $V$ as the matrix $(1)\oplus A$, which we decompose into four $2\times 2$-blocks 
$$
V= 
\begin{pmatrix}
A_{11} & A_{12} \\
A_{21} & A_{22}
\end{pmatrix} = 
\begin{pmatrix}
\mathbf{I}_2 & A_{12} \\
\mathbf{0} & A_{22}
\end{pmatrix}. 
$$

Let $R_{1,1}$ and $R_{2,2}$ be any two elements from $\mt{SL}(2,\Z)$ and let $R_{1,2}$ be the 
$2\times 2$ matrix which solves the equation 
\begin{align*}
A_{12}= R_{12} R_{22}^{-1} (A_{22}-\mathbf{I}_2).
\end{align*}
In other words,
\begin{align}\label{A:subs}
R_{12}= A_{12} (A_{22}-\mathbf{I}_2)^{-1} R_{22}.
\end{align}
Finally, set $R_{21}=\mathbf{0}$ to be the $2\times 2$ matrix whose entries are all 0's, and set
\begin{align*}
R = 
\begin{pmatrix}
R_{11} & R_{12} \\
R_{21} & R_{22}
\end{pmatrix}.
\end{align*}
Note that $R$ is an element of $\mt{SL}(4,\Z)$ and its inverse is given by 
$$
R^{-1} = 
\begin{pmatrix}
R_{1,1}^{-1} & - R_{1,1}^{-1} R_{1,2} R_{2,2}^{-1} \\
\mathbf{0} & R_{2,2}^{-1}
\end{pmatrix}.
$$

Now, 
\begin{align*}
R^{-1} V R &=
\begin{pmatrix}
R_{1,1}^{-1} & - R_{1,1}^{-1} R_{1,2} R_{2,2}^{-1} \\
\mathbf{0} & R_{2,2}^{-1}
\end{pmatrix}
\begin{pmatrix}
\mathbf{I}_2 & A_{12} \\
\mathbf{0} & A_{22} 
\end{pmatrix}
\begin{pmatrix}
R_{1,1} & R_{1,2} \\
\mathbf{0} & R_{2,2}
\end{pmatrix} 
\\
&=
\begin{pmatrix}
R_{1,1}^{-1} & R_{11}^{-1}  A_{12} -   R_{11}^{-1} R_{1,2} R_{22}^{-1}A_{22} \\
\mathbf{0} & R_{2,2}^{-1} A_{22}
\end{pmatrix} 
\begin{pmatrix}
R_{1,1} & R_{1,2} \\
\mathbf{0} & R_{2,2}
\end{pmatrix} \\
&=
\begin{pmatrix}
\mathbf{I}_2 &  R_{11}^{-1} R_{12} +R_{11}^{-1}  A_{12} R_{22}-   R_{11}^{-1} R_{1,2} R_{22}^{-1}A_{22} R_{22} \\
\mathbf{0} & R_{2,2}^{-1} A_{22}R_{22}
\end{pmatrix} 
\end{align*}
We substitute (\ref{A:subs}) into the last matrix and collect the terms:
\begin{align*}
&= 
 \begin{pmatrix}
\mathbf{I}_2 &  R_{11}^{-1} A_{12} ((A_{22}-\mathbf{I}_2)^{-1} + \mathbf{I}_2 + (A_{22}-\mathbf{I}_2)^{-1} A_{22}) R_{22} \\
\mathbf{0} & R_{2,2}^{-1} A_{22}R_{22}
\end{pmatrix} \\
&=
\begin{pmatrix}
\mathbf{I}_2 &  R_{11}^{-1} A_{12} ((A_{22}-\mathbf{I}_2)^{-1}  ( \mathbf{I}_2 - A_{22})+\mathbf{I}_2 ) R_{22} \\
\mathbf{0} & R_{2,2}^{-1} A_{22}R_{22}
\end{pmatrix} \\
&=
\begin{pmatrix}
\mathbf{I}_2 &  \mathbf{0} \\
\mathbf{0} & R_{2,2}^{-1} A_{22}R_{22}
\end{pmatrix}.
\end{align*}

In other words, $V$ is similar to the decomposable matrix 
\begin{align*}
U=
\begin{pmatrix}
\mathbf{I}_2 &  \mathbf{0} \\
\mathbf{0} & R_{2,2}^{-1} A_{22}R_{22}
\end{pmatrix},
\end{align*}
which implies that the fiber bundle $E\times_{S^1} E'$, where $E'$ is the fiber bundle $\bun{S^1}{T^2}{S^1}$, 
is $S^1$-decomposable.

Next, we show that the case 2. of Theorem~\ref{T:our remark} imply that $E$ is stably $S^1$-indecomposable. 
So, we assume that -1 is a root of $f(t)$, the characteristic polynomial of the monodromy matrix $A$ of $E$ but 
$f(1)\neq 0$.


Let $v\in \Z^3$ be an integer (eigen)vector such that 
$A v = -v$. Without loss of generality we assume that the gcd of the entries of $v$ is 1. 
Let $R\in \mt{SL}(3,\Z)$ be the matrix such that 
$$
R v = 
\begin{pmatrix}
1 \\
0 \\
0
\end{pmatrix}.
$$
It is easy to check that $RAR^{-1}$ is of the form 
$$
RAR^{-1} = 
\begin{pmatrix}
-1 & \mathbf{a}' \\
 \mathbf{0} & A_2'
\end{pmatrix},
$$
where $A_2' \in \mt{GL}(2,\Z)$. Obviously, $f(t)$ factors as $(t+1) g(t)$, where $g(t)$ 
is the characteristic polynomial of $A_2'$. Since $f(1)\neq 0$, we see that $g(t)$ is irreducible over integers. 
Indeed, if $g(t)$ is reducible, that is, $g(t)  = (t-a_1)(t-a_2)$ for integers $a_i\in \Z$, then $1=\det A = -a_1a_2$
implies $\{a_1,a_2\} = \{-1,1\}$ which is a contradiction.

Now, assume towards a contradiction that 
there exists a fiber bundle $\bun{F'}{E'}{S^1}$ which stably $S^1$-decomposes $E$. 
Let $B\in \mt{SL}(m,\Z)$ (with $m\leq 3$) denote the monodromy matrix of $E'$. 
(We assume that $m=3$ since the other cases are resolved similarly.)
Then $A\oplus B$ is similar to a direct sum of 3 matrices 
$A \oplus B \sim C_1\oplus C_2 \oplus C_3$ where $C_i\in \mt{SL}(2,\Z)$ for $i=1,2,3$.
Since $g(t)$ is irreducible, 
the characteristic polynomial of one of $C_i$'s is equal to $g(t)$. But the constant term of any matrix 
from $\mt{SL}(2,\Z)$ is 1 whereas the constant term of $g(t)=-1$. It follows from Proposition~\ref{P:decomposable iff}
that $E$ is not stably $S^1$-decomposable.

Finally, we assume that the characteristic polynomial $f(t)$ of $A$ is irreducible. Assume towards a contradiction that 
there exists a fiber bundle $\bun{F'}{E'}{S^1}$ which stably $S^1$-decomposes $E$. 
Let $B\in \mt{SL}(m,\Z)$ (with $m\leq 3$) denote the monodromy matrix of $E'$. 
(We assume that $m=3$ since the other cases are resolved similarly.)
Then $A\oplus B$ is similar to a direct sum of 3 matrices 
$A \oplus B \sim C_1\oplus C_2 \oplus C_3$ where $C_i\in \mt{SL}(2,\Z)$ for $i=1,2,3$.
Since the characteristic polynomial of $A$ is irreducible of degree 3, this decomposition is not possible. 
Therefore, in the light of Proposition~\ref{P:decomposable iff} if the characteristic polynomial of $A$ is irreducible, 
then $E$ is stably $S^1$-indecomposable. 

\end{proof}

\begin{Remark}\label{R:bitirici}
Theorem~\ref{T:stably T3} provides us with an effective way of producing stably $T^m$-decomposable but 
$T^m$-indecomposable fiber bundles for any dimension $m$.
Let $\bun[p_i]{F_i}{E_i}{B_i}$, $i=1,2$ be two fiber bundles over possibly different base spaces $B_1$ and $B_2$. 
The cartesian product $E_1\times E_2$ of total spaces is a fiber bundle over $B_1\times B_2$ with the  projection $p_1 \times p_2$ and fiber 
$F_1\times F_2$. In particular, if $E_1$ denotes the trivial identity fiber bundle $\bun[id]{*}{T^{m-1}}{T^{m-1}}$
and $\bun{T^3}{E_2}{S^1}$ is an $S^1$-indecomposable fiber bundle which is 
stably $S^1$-decomposed by $\bun{F}{E'}{S^1}$, then the product $E=E_1\times E_2$ is a fiber bundle 
over $T^{m}=S^1\times T^{m-1}$ with fiber $T^3$. 
Now $E$ is $T^{m}$-indecomposable, otherwise the existence of such a decomposition would yield a 
decomposition of $E_2$ by restriction to $\{x\}\times E_2 \subseteq E_1 \times E_2$ for any $x \in E_1$.
Also, $E$ is stably $T^{m}$-decomposed by the product bundle $\bun{*\times F}{T^m\times E'}{T^m\times S^1}$. 
Therefore, for any dimension $m\in \Z^+$, we have a $T^3$ bundle over $T^m$ which is $T^m$-indecomposable 
but stably $T^m$-decomposable.
\end{Remark}

\vspace{1cm}
\textbf{Acknowledgement}
We thank Professor Slawomir Kwasik for sharing his manuscript and for his encouragements.
We thank the referee for her/his comments which improved the quality of the paper.

\end{document}